\documentclass[11pt,reqno]{amsart}

\theoremstyle{definition}

\numberwithin{equation}{section}

\newcommand{\be}{\beta}

\theoremstyle{definition}\newtheorem{thm}{Theorem}[section]
\theoremstyle{definition}\newtheorem{cor}[thm]{Corollary}
\theoremstyle{definition}\newtheorem{lem}[thm]{Lemma}
\theoremstyle{definition}\newtheorem{prop}[thm]{Proposition}
\theoremstyle{definition}\newtheorem{defn}[thm]{Definition}
\theoremstyle{definition}\newtheorem{Rem}[thm]{Remark}
\theoremstyle{definition}\newtheorem{exam}[thm]{Example}
\theoremstyle{definition}

\def\be{\begin{enumerate}}
\def\ee{\end{enumerate}}
\usepackage{amssymb}
\begin{document}

\title{Locally closed sets and submaximal spaces}
\author{R. Mohamadian}
\address{Department of Mathematics, Shahid Chamran University of ahvaz, Ahvaz,
Iran}
\email{mohamadian\_r@scu.ac.ir}
\subjclass[2010]{54G99, 54G05, 54G10, 54C30} \keywords{locally closed set, locally indiscrete space, $T_D$-space, submaximal space, Stone-\v{C}ech compactification,  Hewitt realcompactification.}
\begin{abstract}
A topological space $X$ is called submaximal  if every dense subset of $X$ is open. In this paper, we show that if $\beta X$, the Stone-\v{C}ech compactification of $X$,  is a submaximal space, then  $X$ is a compact space and hence $\beta X=X$. We also prove that if  $\upsilon X$, the Hewitt realcompactification of $X$, is submaximal and first countable and $X$ is without isolated point, then $X$ is realcompact and hence $\upsilon X=X$. 
We observe that every submaximal Hausdorff space is pseudo-finite. It turns out that if $\upsilon X$ is a submaximal space, then $X$ is a pseudo-finite $\mu$-compact space. 
 An example is given which shows that $X$ may be submaximal but $\upsilon X$ may not be submaximal. 
Given a topological space $(X,{\mathcal T})$, the collection of all locally closed subsets of $X$ forms a base for a topology on $X$ which is denotes by ${\mathcal T_l}$. We study some topological properties between  $(X,{\mathcal T})$ and $(X,{\mathcal T_l})$, such as we show that  a) $(X,{\mathcal T_l})$ is discrete if and only if  $(X,{\mathcal T})$ is a $T_D$-space; b)  $(X,{\mathcal T})$ is a locally indiscrete space  if and only if ${\mathcal T}={\mathcal T_l}$; c) $(X,{\mathcal T})$ is indiscrete space if and only if $(X,{\mathcal T_l})$ is connected. We see that, in locally indiscrete spaces, the concepts of $T_0$, $T_D$, $T_\frac{1}{2}$, $T_1$, submaximal and discrete coincide.  Finally, we prove that every clopen subspace of an lc-compact space is  lc-compact. 
\end{abstract}
\maketitle
 
\section{Introduction}  

Throughout this paper we consider topological spaces on which  no separation axioms are assumed unless explicity stated. The topology of a space is denoted by ${\mathcal T}$ and $(X,{\mathcal T})$ will be replaced by $X$ if there is no chance for confusion. For a subset $A$ of $(X, {\mathcal T})$, the closure, the interior, the boundary and the set of accumulation points of $A$ are denoted by ${\rm cl}_{\mathcal T}(A)$ or ${\rm cl}_X(A)$, ${\rm int}_{\mathcal T}(A)$ or ${\rm int}_X(A)$, ${\rm Fr}_{\mathcal T}(A)$ or  ${\rm Fr}_X(A)$ and $A'$, respectively. In places where there is no chance for confusion $\overline{A}$ and $A^\circ$ stands for  ${\rm cl}_{\mathcal T}(A)$ and ${\rm int}_{\mathcal T}(A)$, respectively. If $A\subseteq X$ be open, closed or locally closed respect to the topology ${\mathcal T}$ on $X$, then we write sometimes to avoid confusion, ${\mathcal T}$-open, ${\mathcal T}$-closed or ${\mathcal T}$-locally closed, respectively. A subset $A$ of a topological space $X$ is called locally closed if for every $x\in A$ there is an open set $U\subseteq X$ such that $x\in U$ and $A\cap U$ is closed in $U$. Since the intersection of two locally closed sets is locally closed, the family of ${\mathcal T}$-locally closed sets forms a base for a finer topology ${\mathcal T_l}$ on $X$. For a Tychonoff (completely regular Hausdorff) space $X$, $\beta X$ is the Stone-\v{C}ech compactification and $\upsilon X$ is the Hewitt realcompactification of $X$. It is well known that $X$ is compact, pseudocompact or realcompact if and only if $\beta X=X$, $\beta X=\upsilon X$ or $\upsilon X=X$, respectively.  For a Tychonoff space the symbol $C(X)$ (resp. $C^*(X)$) denotes  the ring of all continuous real valued (resp. all bounded continuous real valued) functions defined on  $X$. We say that a subspace $S$ of $X$ is $C^*$-embedded in $X$ if every function in $C^*(S)$ can be extended to a function in $C^*(X)$. For $f\in C(X)$, the zero-set of $f$ is the set $Z(f)=\{x\in X:f(x)=0\}$. The set-theoretic complement of $Z(f)$ is denoted by $coz(f)$. The support of a function $f\in C(X)$ is the $\overline{coz(f)}$. For any $p\in\beta X$, the maximal ideal $M^p$ (resp. the ideal $O^p$) is the set of all $f\in C(X)$ for which $p\in {\rm cl}_{\beta X}Z(f)$ (resp. $p\in{\rm int}_{\beta X}{\rm cl}_{\beta X}Z(f)$). More generally, for $A\subseteq \beta X$, $M^A$ (resp. $O^A$) is the intersection of all $M^p$ (resp. $O^p$) with $p\in A$. It is clear that $O^A\subseteq M^A$.

A space $X$ is said to be a, a) door space if every set is either open or closed; b) $T_\frac{1}{2}$-space if every singleton is either open or closed; c) submaximal space if every dense subset is open; d) principal space, if every intersection of open sets is open, e) $c$-principal space, if every countable intersection of open sets (i.e., $G_\delta$-set) is open, f) resolvable space if it has two disjoint dense subsets. 

Tychonoff $c$-principal spaces are the same $P$-spaces, see Exercise 4J of \cite{GJ}. A door space is submaximal and a submaximal space is $T_\frac{1}{2}$. The one-point compactification of any discrete space is submaximal. A nonempty resolvable space never submaximal. For example, $\Bbb R$ and $\beta \Bbb Q$ are resolvable spaces. For more information about the locally closed sets, see \cite{AC, D, GR}, about the submaximal spaces, see \cite{AC, D, BDE, ABK}, about the door spaces, see \cite{BDE, DLT} and about the $T_\frac{1}{2}$-spaces, see \cite{AT}. For details about $\beta X$ and $\upsilon X$, see \cite{GJ, Wal} and about other concepts of general topology, see \cite{En, Wil}.      

Section 2, of this paper is devoted to investigate some separation properties between  $(X, {\mathcal T})$ and  $(X, {\mathcal T_l})$. Such as we show that  $(X, {\mathcal T_l})$ is discrete if and only if $(X, {\mathcal T})$ is a $T_D$-space.  We show that, in this section, if $\beta X$ is submaximal, then $X$ is compact and therefor, in this case, we conclude that $X=\beta X$.   We also prove that if  $\upsilon X$ is submaximal and first countable and $X$ is without isolated point, then $X$ is realcompact. We observe that every submaximal Hausdorff space is pseudo-finite. It turns out that if $\upsilon X$ is a submaximal space, then $X$ is a pseudo-finite $\mu$-compact space. 
In Section 3, we study and investigate the behavior locally indiscrete spaces. We observe that $(X, {\mathcal T})$ is a locally indiscrete space if and only if ${\mathcal T}={\mathcal T_l}$ if and only if every dense open subset of $X$ is regular open. In Section 4, we introduce some lc-properties such as lc-regular and lc-completely regular and compare them with the concepts regular and completely regular. We prove that every clopen subset of a lc-compact space is lc-compact.

\section{Locally closed sets and ${\mathcal T_l}$-topology}

 A space $X$ is called a $T_D$-space if every singleton is locally closed. It is well known that the space $X$ is submaximal if and only if every subset of $X$ is locally closed, see Theorem 4.2 of \cite{D} and also every subspace of a submaximal space is submaximal, see Theorem 1.1 of \cite{D}.

The proof of the following proposition is straightforward. 

\begin{prop}
The following statements are equivalent, for a subset $A$ of the space $X$.\\
a) $A$ is a locally closed set.\\
b) $A=G\cap B$, which $G$ is open and $B$ is closed.\\
c) $A=H\cap\overline{A}$, which $H$ is open.\\
d) $A=E-F$, which $E$ and $F$ are closed.\\
e) $\overline{A}-A$ is a closed set.\\
f) $A\subseteq (A\cup(X-\overline{A}))^\circ$.\\
g) $A\cup(X-\overline{A})$ is an open set. 	
\end{prop}

\begin{Rem}
a) Every closed (resp. open) set is locally closed.\\
b) Every dense locally closed set is open.\\
c) The intersection of finite number locally closed sets is locally closed.\\
d) Every intersection of locally closed sets  may not be locally closed.\\
e) The complement of a locally closed set need not be locally closed. For example, $A=\{\frac{1}{n}:n=1,\cdots\}$ is a locally closed set in $\Bbb R$ while $\Bbb R-A$ is not locally closed.\\
f) The union of two locally closed sets need not be locally closed. For example, suppose that $A=(-\infty, 0]$ and $B=\bigcup_{n=1}^\infty(\frac{1}{n+1},\frac{1}{n})$ in $\Bbb R$. Then $A\cup B$ is not locally closed.\\
g) The union of two completely separated locally closed sets is locally closed.\\
h) If $A$ is preopen, (that is, $A\subseteq{\rm int}{\rm cl}(A)$) then it is open if and only if it is locally closed.\\
i) If $\overline{A}$ is open, where $A$ is locally closed, then $A$ is open.\\
j) If $A$ is locally closed and $B$ is clopen, then $\overline{A\cap B}$ is locally closed.\\
k) If $A$ and $B$ are locally closed, then the equal $\overline{A\cap B}=\overline{A}\cap\overline{B}$ may not be true. For example, let $X=\{a,b\}$ and ${\mathcal T}=\{\emptyset, \{a\}, X\}$. Now suppose that $A=\{a\}$ and $B=\{b\}$.\\
l) Let $X$ be a topological space and $Y$ a subspace of $X$. The set $A\subseteq Y$ is locally closed in $Y$ if and only if $A=Y\cap B$, where $B$ is locally closed in $X$.\\
m) In a locally compact Hausdorff space, any subset is locally closed if and only if it is a locally compact set, see Theorem 18.4 of \cite{Wil} or Corollary 3.3.10 of \cite{En}.\\
n) Let $X, Y$ are topological spaces,  the continuous mapping $f:X\to Y$ is one-to-one and let the one element family $\{f\}$ separates points and closed sets. If $X$ is a $T_1$-space and $f(X)$ is open subset of $Y$, then $f(X)$ is  a $T_D$-space, see Lemma 2.3.19 of \cite{En}.\\ 
o) A locally closed set $A$ of $X$ is dense in $X$ if and only if there is an open set $G$ of $X$ such that $A\cap G$ is dense in $X$.\\
p) Every discrete subset of a $T_1$-space is locally closed.\\
q)  $A\subseteq X$ is locally closed in $X$ if and only if it is locally closed in $\overline{A}$. 
\end{Rem}

Given a topological space $(X,{\mathcal T})$, the collection of all locally closed subsets of $X$ forms a base for a topology on $X$ which is denotes by ${\mathcal T_l}$. It is clear that ${\mathcal T}\subseteq {\mathcal T_l}$ and in locally indiscrete spaces we have ${\mathcal T}={\mathcal T_l}$, see Proposition \ref{p33}. In the sequel, we study some topological properties between  $(X,{\mathcal T})$ and $(X,{\mathcal T_l})$.

\begin{prop}
$(X,{\mathcal T})$ is indiscrete if and only if $(X,{\mathcal T_l})$ is a connected space.
\end{prop}
\begin{proof}
$(\Rightarrow)$	
Let $G\in {\mathcal T}$ and let $\emptyset\neq G\neq X$. Clearly $G\in{\mathcal T_l}$ and $X-G\in {\mathcal T_l}$ and therefore $(X,{\mathcal T_l})$ is disconnected.\\
$(\Leftarrow)$ Suppose that $X=G\cup H$, which $G$ and $H$ are two nonempty disjoint ${\mathcal T_l}$-open sets and let $x\in G$. Hence there is a ${\mathcal T}$-locally closed set $A$  such that $x\in A\subseteq G$. Therefor $A=U\cap {\rm cl}_{\mathcal T}A$, where $U\in {\mathcal T}$. If $U\neq X$, then we are done. Assume that $U=X$. This implies that $A={\rm cl}_{\mathcal T}A$, that is, $X-A$ is a  ${\mathcal T}$-open. Since $\emptyset\neq H\subseteq X-A$, we infer that $X-A\neq\emptyset$. Now if $X-A=X$, then $A=\emptyset$ which is not true. Consequently, $\emptyset\neq X-A\neq X$ is a  ${\mathcal T}$-open set. This complete the proof.
\end{proof}

\begin{prop}
$(X,{\mathcal T})$ is a $T_0$-space if and only if $(X,{\mathcal T_l})$ is a $T_0$-space.
\end{prop}
\begin{proof}
$(\Rightarrow)$ It is trivial.\\
$(\Leftarrow)$ Let $x,y\in X$ and $x\neq y$. Hence there is $G\in{\mathcal T_l}$ such that $x\in G$ and $y\notin G$. Let $A$ be  ${\mathcal T}$-locally closed set such that $x\in A\subseteq G$. Therefor $A=U\cap {\rm cl}_{\mathcal T}A$, where $U\in {\mathcal T}$. If $y\notin U$, then we are done. If $y\in U$, then $y\notin{\rm cl}_{\mathcal T}A$. Otherwise $y\in A$ and so $y\in G$ which is not true. Thus there is $V\in{\mathcal T}$ such that $y\in V$ and $V\cap A=\emptyset$. This consequence that $x\notin V$ and this complete the proof. 
\end{proof}

If $(X,{\mathcal T})$ is a $T_1$-space then $(X,{\mathcal T_l})$ is discrete. The converse is not true. For example, let $X=\{a,b\}$ and ${\mathcal T}=\{\emptyset, \{a\}, X\}$. For the converse see the next proposition.

\begin{prop}
 $(X,{\mathcal T})$ is a $T_D$-space if and only if $(X,{\mathcal T_l})$ is discrete.	
\end{prop}
\begin{proof}
$(\Rightarrow)$ It is trivial.\\
$(\Leftarrow)$  Let $x\in X$. Hence $\{x\}\in {\mathcal T_l}$, and therefore there is a  ${\mathcal T}$-locally closed set	$A$ such that $x\in A\subseteq \{x\}$. It implies that $A=\{x\}$, i.e.,  $(X,{\mathcal T})$ is a $T_D$-space.
\end{proof}

\begin{Rem}
Every $T_D$-space is $T_0$. The converse is false. For example, we consider the topology ${\mathcal T}=\{(a,\infty):a\in \Bbb R\}\cup\{\emptyset, \Bbb R\}$ on $\Bbb R$.
\end{Rem}
If ${\mathcal T}$ be the principal
topology on $X$, then every element $x\in X$ has a minimal open neighborhood $M_x = \bigcap\{G\in{\mathcal T}: x\in G\}$. For details about principal topology, see \cite{R}.
\begin{lem}
Every principal $T_0$-space is $T_D$.	
\end{lem}
\begin{proof}
Suppose that $x\in X$. If $\overline{\{x\}}=\{x\}$, then we are done. Hence, assume that $x\neq x_0\in\overline{\{x\}}$. Since there is no an open set $G$ such that $x_0\in G$ and $x\notin G$ we infer that there is an open set $H$ such that $x\in H$ and $x_0\notin H$. Now let $U=M_x$ be the minimal open neighborhood of $x$. We claim that $\{x\}=U\cap\overline{\{x\}}$. Assume that $t\in U\cap\overline{\{x\}}$. We show that $t=x$. If not, then two cases occur. In the first case, there is an open set $V$ such that $t\in V$ and $x\notin V$, which is not true, for $t\in \overline{\{x\}}$. In the second case, there is an open set $W$ such that $t\notin W$ and $x\in W$. Now since $t\in U$, we conclude that $t\in W$ which is not true. Consequently, $t=x$ and this implies that $\{x\}$ is locally closed set. 
\end{proof}	

\begin{Rem}
a) If $(X, {\mathcal T})$ is a principal space, then every intersection of locally closed sets is locally closed.\\
b) If  $(X,{\mathcal T})$ is a principal space, then $(X,{\mathcal T_l})$ is principal. The converse is false. For example, 	$(\Bbb R,{\mathcal T_l})$ is principal but $(\Bbb R,{\mathcal T_u})$, where ${\mathcal T_u}$ denotes usual topology on $\Bbb R$, is not principal.
\end{Rem}

\begin{exam}
 A continuous image of a $T_D$-space need not be $T_D$-space. Let $X=(0,1)$ with usual topology and $Y=(1,\infty)$ with ${\mathcal T_Y}=\{Y\cap G:G\in {\mathcal T}\}$, where ${\mathcal T}=\{(a,\infty):a\in \Bbb R\}\cup\{\emptyset, \Bbb R\}$ on $\Bbb R$. We define $f:X\to Y$ by $f(x)=\frac{1}{x}$. Then $f$ is a one-to-one and onto which is continuous. To prove continuity,  suppose that $H\subseteq Y$ be open and $f(x_0)\in H$, where $x_0\in X$. There is $0<r\in \Bbb R$ such that $r<\frac{1}{x_0}$ and $H=(r,\infty)$. We put $G=(s,\frac{1}{r})$, where $s<x_0$. It is clear that $f(G)\subseteq H$. Note that $X$ is a $T_D$-space while $Y$ is not a $T_D$-space.
 
\end{exam}
If $\prod_{\alpha \in\Lambda}X_{\alpha}$ is a $T_D$-space, then $X_{\alpha}$ is a $T_D$-space for each $\alpha\in\Lambda$. The converse is true if $I$ is finite. In the next example we see that an arbitrary product of $T_D$-spaces need not be $T_D$.
\begin{exam}
Let $X_n=\{a,b\}$ and ${\mathcal T_n}=\{\emptyset, \{a\}, X\}$, for any $n\in \Bbb N$ and let $X=\prod_{n \in\Bbb N}X_n$. Suppose that $x=(x_n)\in X$, where $x_n=a$ for every $n\in \Bbb N$. It is clear that $x$ belongs to every nonempty open set in $X$. We claim that $\{x\}$ is not a locally closed set. Otherwise, there is an open set $G$ and a closed set $F$ such that $\{x\}=G\cap F$. If $F\neq X$, then $x\in F\cap(X-F)$ which is not true. If $F=X$, then $\{x\}=G$. But $G=\prod_{n \in\Bbb N}G_n$, where $G_n=X_n$, for every $n\notin I$, which $I\subseteq \Bbb N$ is finite and this is impossible. It is consequence that $\{x\}$ is not locally closed.
\end{exam}

\begin{lem}
 $X=\prod_{\alpha \in\Lambda}X_{\alpha}$ is a $c$-principal space if and only if all but finitely many of the $X_{\alpha}$ are trivial spaces.
\end{lem}
\begin{proof}
It is straightforward.
\end{proof}
In the sequel, by $|A|$ we mean the cardinality of $A$ for every set $A$.
\begin{Rem}
Let	 $X=\prod_{\alpha \in\Lambda}X_{\alpha}$ be a $c$-principal space and $|I|\leq\aleph_0$. Then $X$ is a $T_D$-space if and only if each factor space is a $T_D$-space, see part (b) of Exercise 2.3.B of \cite{En}.
\end{Rem}	

Every $T_\frac{1}{2}$-space is $T_D$. The converse is not true. For example, let $X=\Bbb N$ and ${\mathcal T}=\{E_n:n=1,\cdots\}\cup\{\emptyset\}$, where $E_n=\{n,n+1,\cdots\}$, for any $n\in\Bbb N$. Since $\{n\}=E_n\cap(\Bbb N-E_{n+1})$, we infer that $\Bbb N$ is a $T_D$-space. It is clear that the set $\{2\}$ neither open nor closed, so $\Bbb N$ is not a $T_\frac{1}{2}$-space.

Recall that union of two locally closed sets may not be locally closed. In the next example, we see that even if we add a cluster point to a locally closed set, the resulting set may not be locally closed. In other words,
if $A$ is locally closed and $x\in\overline{A}-A$, then $A\cup\{x\}$ need not be locally closed.
\begin{exam}
Suppose that $X=\Bbb R$,  and for any $k\in \Bbb N$, let $A_k=\{\frac{n+1}{kn+1}:n=1,\cdots\}$. We consider $A=\bigcup_{k=1}^\infty A_k$ and we put $B=A-\{1,\frac{1}{2},\cdots\}$. Then $\overline{B}-B=\{0,1,\frac{1}{2},\cdots\}$ is closed and hence $B$ is locally closed in $\Bbb R$. Now assume that $C=B\cup\{0\}$. Clearly,  $\overline{B}=\overline{C}$. Note that $\overline{C}-C=\{1,\frac{1}{2},\cdots\}$ is not closed set in $\Bbb R$ and hence $C$ is not locally closed.
\end{exam}
	
\begin{exam}
An arbitrary product of Tychonoff, compact and submaximal spaces need not be submaximal. For example, assume that $X_n^*=X_n\cup\{\sigma_n\}$ be a one-point compactification of a discrete space $X_n$, for every $n\in \Bbb N$. Suppose that $X=\prod_{n\in\Bbb N}X_n^*$ and suppose that $A=\prod_{n\in\Bbb N}X_n$, then $A$ is dense in $X$ which is not open. 

\end{exam}

\begin{prop}
Let $X$ be a submaximal space and $A, B\subseteq X$ with $A\cap B=\emptyset$. Then $\overline{A}\cap\overline{B}$ is discrete.
\end{prop}
\begin{proof}
Suppose that $x\in\overline{A}\cap\overline{B}$. Hence $A\cup\{x\}$	is locally closed in $\overline{A}$. Therefor there is an open set $G$ in $\overline{A}$ such that $A\cup\{x\}=G\cap{\rm cl}_{\overline{A}}(A\cup\{x\})$. Let $U$ be open in $X$ such that $G=U\cap \overline{A}$. Note that ${\rm cl}_{\overline{A}}(A\cup\{x\})={\rm cl}_X(A\cup\{x\})\cap\overline{A}=\overline{A}\cap\overline{A}=\overline{A}$. This implies that $A\cup\{x\}=U\cap\overline{A}$. Similarly, there is an open set $V$ in $X$ such that $B\cup\{x\}=V\cap\overline{B}$. Now $\{x\}=(A\cup\{x\})\cap(B\cup \{x\})=(U\cap\overline{A})\cap(V\cap\overline{B})=(U\cap V)\cap(\overline{A}\cap\overline{B})$. This consequence that $x$ is an isolated point of  $\overline{A}\cap\overline{B}$.      
\end{proof}
If  $X$ is a submaximal space and $A\subseteq X$, then ${\rm Fr}(A)$ is a discrete subset of $X$.  Also if $X$ is a submaximal space and $A$ is a discrete subset of $X$, then $\overline{A'}$ is discrete and if in addition $X$ is a $T_1$-space, then $A'$ is discrete.
 
\begin{thm}\label{t1}
If $\beta X$ is a submaximal space, then $X$ is compact. In this case $\beta X=X$.
\end{thm}
\begin{proof}
Since $\beta X$ is countably compact, by Theorem 4.20 of \cite{ABK},  $\beta X$ is a finite disjoint union of one-point compactification of some discrete spaces. Suppose that $\beta X=\bigcup_{i=1}^nX^*_i$, where $X^*_i$ is the one-point compactification  of  discrete space $X_i$, for every $i=1,\cdots,n$. Assume that $X^*_i=X_i\cup\{\sigma_i\}$, $Y=\bigcup_{i=1}^nX_i$ and $A=\{\sigma_1,\cdots,\sigma_n\}$. Clearly, $Y$ is discrete and hence $Y\subseteq X$. It is sufficient to show that $A\subseteq X$. On the contrary suppose that $A\cap(\beta X-X)=I$ is a nonempty set. Then $D=\bigcup_{i\in I}X_i$ is discrete and $X=D\cup(\bigcup_{i\notin I}X^*_i)$. One can easily see that $D$ is $C^*$-embedded in $X$ and consequently  is $C^*$-embedded in $\beta X$. Hence, ${\rm cl}_{\beta X}D=\beta D$ and therefor $\beta D\subseteq\bigcup_{i\in I}X^*_i$. This implies that $|\beta D|\leq|\bigcup_{i\in I}X^*_i|=\sum_{i\in I}|X_i|=|D|$ which is contradiction, for by Theorem 9.2 of \cite{GJ} we have  $|\beta D|=2^{2^{|D|}}$. 
\end{proof}

The converse of the above theorem is not true, in general. For example, $X=\beta \Bbb Q$ is compact, while $\beta X=\beta\Bbb Q$ is not a submaximal space. Also, if $X$ is a submaximal space, then $\beta X$ may not be submaximal. For example, $\Bbb N$ is submaximal but $\beta\Bbb N$ is not submaximal. 
 
\begin{cor}
Let $\upsilon X$ be a submaximal space. The following statements are equivalent.\\
a) $X$ is compact.\\
b) $X$ is $\sigma$-compact.\\
c) $X$ is countably compact.\\
d) $X$ is pseudocompact.
\end{cor}
\begin{proof}
$(a\Rightarrow b)$ and $(c\Rightarrow d)$ are trivial.\\
$(b\Rightarrow c)$ Since $X$ is union a finite number of compact spaces, then it is countably compact.\\
$(c\Rightarrow a)$ If $X$ is a pseudocompact space, then $\upsilon X=\beta X$ and by hypothesis $\beta X$ is submaximal. Now by Theorem \ref{t1} we conclude that $X$ is compact.	
\end{proof}

\begin{lem}
If $X$ is a dense countably compact set in $T$ and $T$ is a Hausdorff submaximal space, then $X=Y$. 
\end{lem}
\begin{proof}
Since $X$ is a submaximal space, then by Theorem 4.21 of \cite{ABK}, $X$ is compact and since $T$ is Hausdorff, we infer that $X$ is closed in $T$. Therefor, by density  of $X$ we have $X=T$. 
\end{proof}
By the above lemma the proof of the next corollary is obvious.
\begin{cor}
If $\upsilon X$ is submaximal and $X$ is countably compact, then $X$ is realcompact.  
\end{cor}

\begin{Rem}
We denote the set of all isolated points of space $X$ with $I(X)$. If $X$ is open in $T$, then $I(X)\subseteq I(T)$ and if $X$ is dense in $T$, then $I(T)\subseteq I(X)$. Hence, if $X$ is an open dense set in $T$, then $I(X)=I(T)$.	
\end{Rem}
	
\begin{thm}\label{t2}
Suppose that $T$ is a first countable, Hausdorff and submaximal space and $X$ is a dense set in $T$ and $I(X)=\emptyset$. Then $X=T$.  	
\end{thm}
\begin{proof}
Let $p\in T-X$. Hence, there is an infinite sequence $(x_n)$ contained in $X$ which converges to $p$. Put $A=\{x_n:n=1,\cdots\}\cup\{p\}$. Let $N$ be a copy of $\Bbb N$ contained in $A$. Since $A$ is compact, then $N$ has an accumulation point in $A$, namely $x_0$. We claim that $\overline{T-N}=T$. To see this, let $a\in T$ and $a\in G\subseteq T$, which $G$ is an open set in $T$. We must show that $G\cap (T-N)\neq\emptyset$. If not, then $a\in G\subseteq N$ and hence there is an open set $H$ in $T$ such that $H\cap N=\{a\}$. It is clear that $H\cap G=\{a\}$, that is $a\in I(T)$ which is contradiction. Now one can easily see that $x_0\in T-N$ while $x_0\notin (T-N)^\circ$, that is $T-N$ is not open, so it contradicts the submaximality  of $T$.  
\end{proof}

\begin{cor}
Let $X$ be a compact submaximal space. Then $|I(X)|$ is finite if and only if $X$ is finite.
\end{cor}
\begin{proof}
Suppose that $|I(X)|$ is finite. If $X-I(X)$ is infinite, then similar to the proof of the Theorem \ref{t2}, we can show that $X$ is not submaximal, which is a contradiction. Hence $X-I(X)$ is finite and so $X$ is finite.
\end{proof}

\begin{cor}
If $\upsilon X$ is submaximal and  first countable and $I(X)=\emptyset$, then $X=\upsilon X$ i.e., $X$ is realcompact.  
\end{cor}
\begin{proof}
By Theorem \ref{t2}	is clear.
\end{proof}

A space $X$ is called a pseudo-finite (resp. pseudo-discrete) space if every compact subspace of $X$ is finite (resp. has finite interior). We say that $X$ is real pseudo-finite, if $\upsilon X$ is pseudo-finite. Every  pseudo-finite space is pseudo-discrete, but not conversely. For example, we consider the space $\Bbb Q$ of rational numbers.

\begin{exam}
The space $\Sigma$ is pseudo-finite. To see this, let $F\subseteq\Sigma$ be a compact subspace. If $\sigma\notin F$, then $F$ is  discrete and hence it is finite. If $\sigma\in F$ and $F$ is infinite, then assume that $A=\{x_1, x_2,\cdots\}$ be an infinite subset of $F$ and put $G=(\Bbb N-\{x_1\})\cup\{\sigma\}$. Now ${\mathcal C}=\{\{x_n\}:n=1,\cdots\}\cup\{G\}$ be an open cover for $F$ which has not a finite subcover. This is a contradiction.  
\end{exam}

\begin{prop}\label{p44}
Every submaximal Hausdorff space is pseudo-finite.
\end{prop}
\begin{proof}
Let $X$ be a submaximal Hausdorff space and let $Y$ be a compact subspace of $X$. If $Y$ be infinite, then contains a copy of $\Bbb N$, namely $N$. Suppose that $x_0\in N'$, where $x_0\in Y$. Now similar to the proof of the Theorem \ref{t2}, it is observe that $X-N$ is dense in $X$ but it is not an open set in $X$, so it contradicts the submaximality of $X$.	
\end{proof}
	
A  pseudo-finite Tychonoff space may be not submaximal. For example, $\Bbb Q$ is  pseudo-finite but not submaximal. See also the following example.

\begin{exam}
We consider the space $X$ in Example 2 of \cite{RW}. This example shows that the Tychonoff space $X$ is an infinite countably compact subset of $\beta\Bbb N$ 	which is also pseudo-finite. We claim that $X$ is not submaximal. Otherwise, by Theorem 4.20 of \cite{ABK}, $X$ must be a compact space. In this case, since $X$ is pseudo-finite, it must be finite, which is not true. 
\end{exam}	

$C_K(X)$ (resp. $C_\psi(X)$) is the family of all $f\in C(X)$ with compact (resp. pseudocompact) support. Also $C_{\infty}(X)$ is a subcollection of $C^*(X)$ consisting of all functions vanishing at infinity, that is all $f\in C(X)$ for which $\{x\in X:|f(x)|\geq\frac{1}{n}\}$ is compact for every $n\in\Bbb N$. For the convenience of readers, some special ideals are listed below.

a) $C_F(X)=O^{\beta X\setminus I(X)}$ is the intersection of all essential ideals in $C(X)$. Recall that $C_F(X)$ is the socle of $C(X)$ and it is well known  that $C_F(X)=\{f\in C(X):coz(f)~\mbox{is finite}\}$, see Proposition 3.3 of \cite{KR}.

b) $C_K(X)=O^{\beta X\setminus X}$ is the intersection of all free ideals in $C(X)$, see 7E of \cite{GJ}.

c) $C_\infty(X)=M^{*\beta X\setminus X}$ is the intersection of all free maximal ideals in $C^*(X)$.

d) $C_\psi(X)=M^{\beta X\setminus \upsilon X}$ is the intersection of all hyper-real maximal ideals in $C(X)$.

e) $I_\psi(X)=M^{\beta X\setminus X}$ is the intersection of all free maximal ideals in $C(X)$.\\

One can easily see that $C_F(X)\subseteq C_K(X)\subseteq I_{\psi}(X)$. In Theorem 8.19 of \cite{GJ}, it is shown that if $X$ is realcompact, then $C_K(X)=I_{\psi}(X)$. In part (a) of Theorem 4.5 of \cite{Az}, it is proved that $X$ is a pseudo-discrete if and only if $C_F(X)=C_K(X)$.

In \cite{JM} a space $X$ is called:\\
a) a $\mu$-compact if $C_K(X)=I_{\psi}(X)$.\\
b) an $\eta$-compact if $C_{\psi}(X)=I_{\psi}(X)$.\\
c) a $\psi$-compact if $C_K(X)=C_{\psi}(X)$.\\
d) an $\infty$-compact if $C_K(X)=C_{\infty}(X)$.

In \cite{M} a space $X$ is called an $i$-compact if $C_{\infty}(X)=I_{\psi}(X)$.

\begin{prop}
If $\upsilon X$ is a submaximal space, then $X$ is a pseudo-finite $\mu$-compact space.	
\end{prop}	
\begin{proof}
By Corollary \ref{p44}, $X$ is 	pseudo-finite. Since $X$ is real  pseudo-finite, by Theorem 3.8 of \cite{KR}, we infer that $C_K(X)=I_\psi(X)$, that is, $X$ is a $\mu$-compact space.
\end{proof}	
	
The converse of the above proposition is not true, in general. For example, $\Bbb Q$ is a pseudo-finite $\mu$-compact space but $\upsilon\Bbb Q=\Bbb Q$ is not submaximal.\\
By part (b) of Theorem 4.5 of \cite{Az}, the proof of the result is obvious.
\begin{cor}
Let $X$ be a submaximal Hausdorff space and $|I(X)|<\aleph_0$. Then $C_K(X)=C_{\infty}(X)$.	
\end{cor}

\begin{prop}
Let $\upsilon X$ be a submaximal space and $|I(X)|<\aleph_0$. Then the following statement are hold.\\
a) $X$ is a $i$-compact.\\
b) $X$ is a $\infty$-compact.\\
c) $X$ is compact if and only if it is a realcompact space.
\end{prop}
\begin{proof}
The proof of parts (a) and (b) and implication ($\Rightarrow$) of (c) is clear. For the other side of (c), since $X$ is $i$-compact, then $C_{\infty}(X)=I_{\psi}(X)$. Since $X$ is realcompact, we infer that $C_{\psi}(X)=I_{\psi}(X)=C_K(X)$, that is, $X$ is $\eta$-compact and $\psi$-compact. Hence $C_{\infty}(X)=C_{\psi}(X)$  and by Corollary 2.5 of \cite{AS} we conclude that $X$ is compact.	
\end{proof}

Similar to Theorem \ref{t1}, the question arises that if $\upsilon X$ is a submaximal space, is $X$ realcompact?. So far, we have not been able to answer this question, in general. Therefor we express it as a question.\\

Question:  If $\upsilon X$ is a submaximal space, is $X$ realcompact?\\
 
The space $\Bbb R$ is a realcompact space which is not submaximal. See the following example for an example of a space that is submaximal but not realcompact. This example, also shows that $X$ may be submaximal but $\upsilon X$ may not be submaximal.

\begin{exam}
We consider the space ${\bold\Psi}$ in 5I of \cite{GJ}. By 5I. 5, the space ${\bold\Psi}$ is pseudocompact which it is not a realcompact space. We now show that it is submaximal. Suppose that $\overline{B}={\bold\Psi}$. Since $\Bbb N\subseteq B$, we infer that $B=\Bbb N\cup A$, where $A=\{\omega_E:E\in{\mathcal F\subseteq{\mathcal E}}\}$. Assume that $\omega_E\in B$, which $E\in{\mathcal F}$. Put $G=D\cup\{\omega_E\}$, where $D$ containing all but a finite number of points of $E$. Then $G$ is an open set contains $\omega_E$ and $G\subseteq B$. This implies that $\omega _E\in B^\circ$, that is  $B$ open and consequence that ${\bold\Psi}$ is submaximal. Furthermore, since ${\bold\Psi}$ is not compact, by Theorem \ref{t1}, $\beta{\bold\Psi}$  is not submaximal. Now $\upsilon{\bold\Psi}=\beta{\bold\Psi}$ shows that $\upsilon{\bold\Psi}$ is not a submaximal space.
\end{exam}

It is possible that  $X$ is a locally compact and pseudocompact space but $\upsilon X$ is not submaximal. See the next example.  

\begin{exam}
We consider the space $\bold W=W(\omega_1)=\{\sigma:\sigma<\omega_1\}$ of all countable ordinals and $\bold W^*=W(\omega_1+1)=\{\sigma:\sigma\leq\omega_1\}$, where $\omega_1$ denoting the first uncountable ordinal. It is well known that $\bold W$ is a pseudocompact locally compact space which neither compact nor realcompact. Clearly, $\upsilon \bold W=\beta \bold W=\bold W^*$. Since $\bold W$ is not compact, then by Theorem \ref{t1},  $\upsilon \bold W$
 is not submaximal. For a direct proof let $A=\{\omega_0, 2\omega_0,3\omega_0.\cdots\}$, where $\omega_0$ is the first infinite ordinal and let $B=\bold W^*-A$. One can easily check that $B$ is dense in $\upsilon\bold W$. But $B$ is not open in $\upsilon\bold W$, for $\omega_0^2\in B$ while $\omega_0^2\notin B^\circ$. 
\end{exam}

\section{locally indiscrete spaces}

A space $X$ is called locally indiscrete if  every open set is closed or equivalently if every closed set is open. Every discrete space is locally indiscrete. For another nontrivial example, let $R$ be a principal ideal ring. Then the space ${\rm Min}(R)$ with Zariski topology is a locally indiscrete space.  

\begin{prop}
For a topological space $X$ the following conditions are equivalent.\\
a) $X$ is locally indiscrete.\\
b) Every subset of $X$ is preopen.\\
c) Every singleton in $X$ is preopen.\\
d) Every closed subset of $X$ is preopen.\\
e) Every locally closed subset of $X$ is open.\\
f) Every locally closed subset of $X$ is closed.\\
g) The closure of every locally closed set is open.\\
h) Every dense open subset of $X$ is regular open.	
\end{prop}
\begin{proof}
All implications are straightforward. We only show $(h\Rightarrow a)$. Suppose that $A$ is an open set in $X$ and consider $B=A\cup(X-A)^\circ$. Then $B$ is open in $X$. Furthermore, $\overline{B}=\overline{A}\cup\overline{(X-A)^\circ}=\overline{A}\cup\overline{(X-\overline{A})}=\overline{A}\cup (X-\overline{A}^\circ)=X$. Hence, by hypothesis $B$ is regular open and therefore we have $X=\overline{B}^\circ\subseteq B$. It implies that $X=A\cup(X-\overline{A})$ and thus $\overline{A}\cap(X-A)=\emptyset$. This consequence that $\overline{A}\subseteq A$, that is, $A$ is closed and we are done. 
\end{proof}
\begin{prop}
For a topological $T_1$-space $X$ the following conditions are equivalent.\\
a) $X$ is discrete.\\
b) $X$ is locally indiscrete.\\
c) Every open set is regular open.
\end{prop}
\begin{proof}
It is straightforward. 
\end{proof}

To see the definition of the concepts given in the next remark, see \cite{GJ, En}. Foe details about $P_F$-spaces, see \cite{AMM}.

\begin{Rem}
a) A locally indiscrete space need not be discrete.\\
b) If $X$ is a 	$T_\frac{1}{2}$-space, then it is locally indiscrete if and only if it is discrete.\\
c) Every  locally indiscrete space is a strongly zero-dimensional space. The converse is not true. For example, $A=\{0,1,\frac{1}{2},\cdots\}$ is a strongly zero-dimensional subspace of $\Bbb R$ while it is not  locally indiscrete.\\
d) Every  locally indiscrete space is a $P_F$-space. The converse is not true. For example, the space $\Sigma$ of 4M in \cite{GJ} is a $P_F$-space which is not locally indiscrete.\\
e) A locally indiscrete space need not be submaximal. For instance, let $X=\{a,b,c\}$ and ${\mathcal T}=\{\emptyset, \{a\}, \{b,c\}, X\}$.\\ 
f) A submaximal space need not be a locally indiscrete space. For example we consider the space $\Sigma$ of 4M in \cite{GJ}.\\
g) Every  locally indiscrete space is a $P$-space. The converse is false. For example, the space $S$ in 4.N of \cite{GJ} is a $P$-space which is not locally indiscrete.
h) Every  locally indiscrete space is a extremally disconnected and hence is basically disconnected. The converse is false. For example, the space $S$ in 4.N of \cite{GJ} is a basically disconnected which is not locally indiscrete and the space $\Sigma$ is extremally disconnected which is not locally indiscrete. 	
\end{Rem}

\begin{prop}\label{p33}
The space  $(X,{\mathcal T})$ is a locally indiscrete space  if and only if ${\mathcal T}={\mathcal T_l}$.
\end{prop}
\begin{proof}
$(\Rightarrow)$	Let $G\in {\mathcal T_l}$. Then $G=\bigcup_{\alpha\in \Lambda} A_{\alpha}$, where $A_{\alpha}$ is ${\mathcal T}$-locally closed, for any $\alpha\in\Lambda$. By hypothesis, $A_{\alpha}$ is ${\mathcal T}$-open, hence $G\in {\mathcal T}$ and we are done.\\
$(\Leftarrow)$ Let $A$ is a ${\mathcal T}$-locally closed subset of $X$. Hence $A\in {\mathcal T_l}$ and therefore $A\in{\mathcal T}$. This shows that $X$ is a locally indiscrete space.
\end{proof}

\begin{prop}
Let $X$ be a locally indiscrete space. The following conditions are equivalent.\\
a) $X$ is a $T_1$-space.\\
b) $X$ is a $T_\frac{1}{2}$-space.\\
c) $X$ is a $T_D$-space.\\
d) $X$ is a $T_0$-space.\\
e) $X$ is a submaximal space.\\
f) $X$ is a discrete space.
\end{prop}
\begin{proof}
All implications are obvious. We only show $(d\Rightarrow a)$. Let $x\in X$ and on the contrary suppose that $y\in\overline{\{x\}}$ and $y\neq x$. If there is an open set $G$ such that $y\in G$ and $x\notin G$, then $G\cap \{x\}\neq\emptyset$ which is not true. If there is an open set $H$ such that $x\in H$ and $y\notin H$, then $y\in X-H$ and since $X-H$ is open we infer that  $(X-H)\cap \{x\}\neq\emptyset$ and this is not true. Therefor $X$ is a $T_1$-space and we are through.
\end{proof}

\section{lc-properties}
In this section, by using the locally closed sets, we introduce some separation axioms. For more details about lc-continuous functions, see \cite{GR}. We begin with the following definition.

\begin{defn}
A space $X$ is called:\\
a) lc-regular if for each locally closed set $A$ and for each point $x\notin A$, there are disjoint open sets $U$ and $V$ with $x\in U$ and $A\subseteq V$.\\ 
b) lc-completely regular if for each locally closed set $A$ and for each point $x\notin A$, there exists a continuous function $f:X\to [0,1]$ such that $f(x)=0$ and $f(A)=1$.\\
c) lc-normal if for every two disjoint locally closed sets $A$ and $B$, there are disjoint open sets $U$ and $V$ with $A\subseteq U$ and $B\subseteq V$.
\end{defn}

Every lc-regular (resp. lc-completely regular, lc-normal)  space is a regular  (resp. completely regular, normal) space. The converse is hold in locally indiscrete spaces, but is not true, in general. See the following example.
\begin{exam}
We consider $X=\Bbb R$ with usual topology.\\
a) $\Bbb R$ is not a lc-regular space. To see this let $A=[0,1)$. Then $A$ is locally closed and $1\notin A$. But $A$ cannot be separated from $1$ by disjoint open sets.\\
b) $\Bbb R$ is not a lc-completely regular space. To see this let $A=\{1,\frac{1}{2},\cdots\}$. Then $A$ is locally closed and $0\notin A$. But $A$ cannot be separated from $0$ by a continuous function.\\
c) $\Bbb R$ is not a lc-normal space. To see this let $A=\{1,\frac{1}{2},\cdots\}$ and $B=\{0\}$. Then $A$ and $B$ are locally closed  sets. But $A$ and $B$ cannot be separated by disjoint open sets.
\end{exam}  

\begin{Rem}
a) Every lc-completely regular $T_D$-space is completely regular.\\
b) Every lc-normal $T_D$-space is Hausdorff.
\end{Rem}

\begin{defn}
A space $X$ is called lc-compact if each locally closed cover of $X$ has a finite subcover.
\end{defn}
Every lc-compact space is compact. Every infinite compact $T_1$-space is not lc-compact. One can easily see that $(X,{\mathcal T})$ is lc-compact if and only if   $(X,{\mathcal T_l})$ is compact.

\begin{prop}
Every clopen subset of a lc-compact space is lc-compact.
\end{prop}
\begin{proof}
Suppose that $Y\subseteq\bigcup_{\alpha\in \Lambda}A_\alpha\in\Lambda$, where $A_\alpha$ is locally closed set in $Y$, for each $\alpha\in \Lambda$. Hence, $A_\alpha=B_\alpha\cap Y$, which $B_\alpha$ is locally closed set in $X$, for each $\alpha\in\Lambda$. Therefor there is an open set $G_\alpha$ and a closed set $F_\alpha$ in $X$ which $B_\alpha=G_\alpha\cap F_\alpha$. Now it is clear that $X=\bigcup_{\alpha\in \Lambda}((G_\alpha\cup(X-Y))\cap(F_\alpha\cup(X-Y)))$. Since $X$ is a lc-compact space we infer that $X=\bigcup_{k=1}^n((G_{\alpha_k}\cup(X-Y))\cap(F_{\alpha_k}\cup(X-Y)))$, for a natural number $n$ and $\alpha_i\in\Lambda$ for $i=1,\cdots,n$ . It implies that $Y\subseteq\bigcup_{k=1}^n A_{\alpha_k}$, that is, $Y$ is a lc-compact space.
\end{proof}

\begin{defn}
A function $f:X\to Y$ is called lc-continuous if the converse image of any open set in $Y$ is locally closed in $X$.
\end{defn}

Every continuous function is lc-continuous. The converse is not true. For example, let $X=\{a,b\}$ and  ${\mathcal T_1}=\{\emptyset, \{a\}, X\}$ and ${\mathcal T_2}=\{\emptyset, \{a\}, \{b\}, X\}$ are two topology on $X$. We define $f:(X,{\mathcal T_1})\to (X,{\mathcal T_2})$ by $f(x)=x$. Then $f$ is lc-continuous but it is not a continuous function.

\begin{Rem}
Let $f:X\to Y$ be a onto and $X$ be a lc-compact space. Then\\
a) if $f$ is continuous then $Y$ is lc-compact.\\
b)  if $f$ is lc-continuous then $Y$ is compact.
\end{Rem}
\begin{Rem}
If $f:(X,{\mathcal T})\to Y$ is lc-continuous, then 	 $f:(X,{\mathcal T_l})\to Y$ is continuous. The converse is false. To see this we consider $(\Bbb R, {\mathcal T_u})$ and define $f:(\Bbb R, {\mathcal T_l})\to Y=\{0,1\}$ by $f(\Bbb Q)=0$ and $f(\Bbb R-\Bbb Q)=1$, which $Y$ is equipped with discrete topology. Since $(\Bbb R, {\mathcal T_l})$ is discrete, then $f$ is continuous. But $\Bbb Q=f^{-1}(\{0\})$ is not locally closed set in $(\Bbb R, {\mathcal T_u})$ while $\{0\}$ is an open set in $Y$. This shows that $f:(\Bbb R, {\mathcal T_u})\to Y$ is not a lc-continuous function.
\end{Rem}

\begin{Rem}
Let  $f:X\to Y$ be continuous.Then\\
a) if $B\subseteq Y$ is locally closed then $f^{-1}(B)\subseteq X$ is locally closed.\\
b) if $A\subseteq X$ is locally closed then $f(A)$ need not be locally closed. Let $X=\{a,b\}$,  ${\mathcal T_1}$ be discrete topology and  ${\mathcal T_2}$ be trivial topology on $X$. We define $f:(X,{\mathcal T_1})\to (X,{\mathcal T_2})$ by $f(x)=x$. Then $A=\{a\}$ is locally closed set in $(X,{\mathcal T_1})$ but $f(A)=\{a\}$ is not locally closed set in $(X,{\mathcal T_2})$. 
\end{Rem}

\begin{defn}
A function $f:X\to Y$ is called locally closed if the image of each locally closed set of $X$, is  locally closed in $Y$.
\end{defn}
If  $f:X\to Y$ is one-to-one, open and closed function, then it is locally closed. A locally closed function need not be open or closed. See the next example.
\begin{exam}
a) Let $X=\{a,b\}$,  ${\mathcal T_1}$ be discrete topology and  ${\mathcal T_2}=\{\emptyset,\{a\}, X\}$ be another topology on $X$.  We define $f:(X,{\mathcal T_1})\to (X,{\mathcal T_2})$ by $f(x)=x$. Then $f$ is locally closed but it is not an open function.\\
b)  Let $X=\{a,b,c\}$,  ${\mathcal T_1}$ be discrete topology and  ${\mathcal T_2}=\{\emptyset,\{a,c\}, \{b,c\}, \{c\}, X\}$ be another topology on $X$.  We define $f:(X,{\mathcal T_1})\to (X,{\mathcal T_2})$ by $f(x)=x$. Then $f$ is locally closed but it is not a closed function.
	
\end{exam}


\end{document}